\documentclass[10pt,a4paper]{article}
\usepackage[T1]{fontenc}
\usepackage[english]{babel}
\usepackage{amsmath}
\usepackage{amsfonts}
\usepackage{amssymb}
\usepackage{amsthm}
\usepackage{titlesec}
\usepackage{indentfirst}
\usepackage{color}
\usepackage{enumerate}

\newtheorem{lemme}{Lemma}

\newtheorem{theo}{Theorem}
\newtheorem*{theoreme}{Main Theorem}

\newtheorem{proposition}{Proposition}

\theoremstyle{definition}

\newtheorem*{rem}{Remark}

\newtheorem*{ex}{Example}

\title{Isometric and invertible composition operators on weighted Bergman spaces of Dirichlet series}
\date{}
\newcommand\dis{\displaystyle}

\newcommand{\biindice}[3]%
{%

\begin{array}[t]{c}
{\displaystyle #1}\\
{\scriptstyle #2}\\
{\scriptstyle #3}
\end{array}

}

\titleformat{\section}[block]
{\normalfont\large\bfseries\filcenter}
{\thesection}
{1em}
{}

\titleformat{\subsection}[block]
{\normalfont\bfseries\filcenter}
{\thesubsection}
{1em}
{}

\begin{document}
\begin{center}
{\Large \bfseries Isometric and invertible composition operators on weighted Bergman spaces of Dirichlet series}
\vspace{5pt}

{ \textit{ Maxime Bailleul }

 				   \par\vspace*{20pt}} 
\end{center}

\begin{abstract} We show that a composition operator on weighted Bergman spaces $\mathcal{A}_{\mu}^p$ is invertible if and only if it is Fredholm if and only if it is  an isometry. 
\end{abstract}

%
%

\section{Introduction}
In \cite{hedenmalm1995hilbert}, the authors defined the Hardy space $ \mathcal{H}^2$ of Dirichlet series with square-summable coefficients. Thanks to the Cauchy-Schwarz inequality, it is easy to see that $\mathcal{H}^2$ is a  space of analytic functions on $ \mathbb{C}_{ \frac{1}{2}}:= \lbrace s \in \mathbb{C}, \, \Re(s) > \frac{1}{2} \rbrace$. \linebreak F. Bayart introduced in \cite{bayart2002hardy} the more general class of Hardy spaces of Dirichlet series $ \mathcal{H}^p$ ($1 \leq p < + \infty$). In another direction, McCarthy defined in \cite{mccarthy2004hilbert} some weighted Bergman Hilbert spaces of Dirichlet series and these spaces have been generalized in \cite{bailleullefevre2013}.

In order to recall how these spaces are defined, we need to recall the principle of the Bohr's point of view: let ${n\geq 2}$ be an integer, it can be written (in a unique way) as a product of prime numbers $n=p_1^{ \alpha_1} \cdot \cdot \, \, p_k^{ \alpha_k}$ where $p_1=2
, \, p_2=3 $ \linebreak etc \dots For $s \in \mathbb{C}$, we consider $z=(z_1, \, z_2, \dots) = (p_1^{-s}, \, p_2^{-s}, \dots )$. Then, writing
\begin{equation} f(s) = \sum_{n=1}^{+ \infty} a_n n^{-s} \end{equation}
we get
\[ f(s) = \sum_{n=1}^{ + \infty } a_n (p_1^{-s})^{ \alpha_1} \cdot \cdot \, \, (p_k^{-s})^{ \alpha_k} = \sum_{n=1}^{ + \infty } a_n \, z_1^{ \alpha_1} \cdot \cdot \, \, z_k^{ \alpha_k}.  \]
So we can see a Dirichlet series as a Fourier series on the infinite-dimensional polytorus $ \mathbb{T}^{ \infty}= \lbrace (z_1, z_2, \cdots), \, \vert z_i \vert=1, \, \forall i \geq 1 \rbrace$. We shall denote this Fourier series $D(f)$. 

Let us fix now $p\ge1$. The space $H^p( \mathbb{T}^{ \infty})$ is the closure of the set of analytic polynomials with respect to the norm of $L^p( \mathbb{T}^{ \infty}, \, m)$ where $m$ is the normalized Lebesgue measure on $\mathbb{T}^{\infty}$. Let $f$ be a Dirichlet polynomial, $D(f)$ is then an analytic polynomial on $ \mathbb{T}^{ \infty}$ by the Bohr's point of view. By definition, $ \Vert f \Vert_{ \mathcal{H}^p}:= \Vert D(f) \Vert_{ H^p( \mathbb{T}^{ \infty})}$ and $ \mathcal{H}^p$ is the closure of the set of Dirichlet polynomials with respect to this norm. The spaces $ \mathcal{H}^p$ and $ H^p( \mathbb{T}^{ \infty})$ are then isometrically isomorphic. \\

 We recall now how we can define the weighted Bergman spaces of Dirichlet series. For $ \sigma>0$, $f_{ \sigma}$ will be the translate of $f$ by $\sigma$, \textit{i.e.} ${f_{ \sigma}(s):= f ( \sigma+s)}$. We shall denote by $\mathcal{P}$ the set of Dirichlet polynomials.

Let $p \geq 1$, $P \in \mathcal{P}$ and $ \mu$ be a probability measure on $(0, + \infty)$ such that $0 \in Supp( \mu)$. Then 
\[ \Vert P \Vert_{ \mathcal{A}_{ \mu}^p}:= \bigg{(} \int_{0}^{+ \infty} \Vert P_{ \sigma} \Vert_{ \mathcal{H}^p}^p \, d \mu( \sigma) \bigg{)}^{1/p} \cdot \]
$\mathcal{A}_{\mu}^p$ is the completion of $\mathcal{P}$ with respect to this norm. When $ d\mu( \sigma) = 2 e^{-2 \sigma} \, d \sigma$, these spaces are simply denoted by $ \mathcal{A}^p$. It is shown in \cite{bailleullefevre2013} that they are spaces of convergent Dirichlet series on $\mathbb{C}_{1/2}$. \\

In \cite{gordon1999composition}, the bounded composition operators on $ \mathcal{H}^2$, in other words the analytic functions $\Phi: \mathbb{C}_{ \frac{1}{2}} \rightarrow \mathbb{C}_{ \frac{1}{2}}$ such that for any $f \in \mathcal{H}^2$, $f \circ \Phi \in \mathcal{H}^2$, are characterized. In \cite{bayart2002hardy}, F. Bayart generalized this result to the space $\mathcal{H}^p$ when $p \geq 1$.

We denote by $\mathcal{D}$ the set of functions $f$ which admit a representation by a convergent Dirichlet series in some half-plane and for $\theta \in \mathbb{R}$, $\mathbb{C}_{\theta}$ will be the following half-plane $ \lbrace s \in \mathbb{C}, \, \Re(s) > \theta \rbrace$. We shall denote $\mathbb{C}_+$ instead of $\mathbb{C}_0$.

On the spaces $\mathcal{A}_{\mu}^p$, the following theorems have been proved in \cite{bailleulcompo}:
\begin{theo}[\cite{bailleulcompo},Th1] Let $ \Phi: \mathbb{C}_{ \frac{1}{2}} \rightarrow \mathbb{C}_{ \frac{1}{2}}$ be an analytic function of the form $ \Phi(s) =c_0s + \varphi(s)$ where $c_0 \geq 1$ and $ \varphi \in \mathcal{D}$. Then $C_{\Phi}$ is bounded on $\mathcal{A}_{\mu}^p$ if and only if $ \varphi$ converges uniformly in $ \mathbb{C}_{\varepsilon}$ for every $ \varepsilon>0$ and $ \varphi( \mathbb{C}_+) \subset \mathbb{C}_+$. Moreover in this case, $C_{\Phi}$ is a contraction.
\end{theo}

\begin{theo}[\cite{bailleulcompo},Th2] Let $ \Phi: \mathbb{C}_{ \frac{1}{2}} \rightarrow \mathbb{C}_{ \frac{1}{2}}$ be in $\mathcal{D}$. Then
\begin{enumerate}[(i)]
\item
If $C_{\Phi}$ is bounded on $\mathcal{A}_{\mu}^p$ then $ \Phi$ converges uniformly in $ \mathbb{C}_{\varepsilon}$ for every $ \varepsilon>0$ and $ \Phi( \mathbb{C}_+) \subset \mathbb{C}_{1/2}$.
\item
If  $\Phi$ converges uniformly in $ \mathbb{C}_{\varepsilon}$ for every $ \varepsilon>0$ and $ \Phi( \mathbb{C}_+) \subset \mathbb{C}_{1/2+ \eta}$ with some $\eta>0$ then $C_{\Phi}$ is bounded on $\mathcal{A}_{\mu}^2$.
\end{enumerate}
\end{theo}

In the sequel, we assume that $\mu$ is a probability measure on $(0, + \infty)$ such that $d \mu( \sigma) =h( \sigma) d \sigma$ where $h$ is a positive continuous function.

\begin{ex}
Let $\alpha>-1$, we denote $\mu_{\alpha}$ the probability measure defined on $(0, + \infty)$ by
\[ d\mu_{\alpha}( \sigma) = \frac{2^{\alpha+1}}{\Gamma(\alpha+1)} \sigma^{\alpha} e^{-2 \sigma} \,  d \sigma.\]
We denote the corresponding space $\mathcal{A}_{\alpha}^p$ instead of $ \mathcal{A}_{\mu_{\alpha}}^p$.
\end{ex}

\begin{theoreme}{\label{Thprincipal}}
Let $1 \leq p < + \infty$ and $C_{\Phi}$ be a bounded composition operator on $\mathcal{A}_{\mu}^p$. The following assertions are equivalent:
\begin{enumerate}[(i)]
\item $C_{\Phi}$ is inversible.
\item $C_{\Phi}$ is Fredholm.
\item $C_{\Phi}$ is an isometry.
\item $\Phi$ is a vertical translation: there exists $\tau \in \mathbb{R}$ such that for every $s \in \mathbb{C}_+$, $\Phi(s)=s+i \tau$.
\end{enumerate}
\end{theoreme}

We point out that the result is false on the spaces $\mathcal{H}^p$: F. Bayart proved that $(i),(ii),(iii)$ are still equivalent on $\mathcal{H}^p$ but obtained a different characterization for the isometric composition operators on $\mathcal{H}^p$ (see \cite{bayart2002hardy}). For example, if $\Phi$ is defined for every $s \in \mathbb{C}_+$ by $\Phi(s)=c_0s$ with $c_0 \geq2$, then $C_{\Phi}$ is an isometry on $\mathcal{H}^p$ but not on $\mathcal{A}_{\mu}^p$. The same phenomenon appears in the framework of composition operators on the unit disk (see \cite{martin2006isometries}). \\

Ir order to prove the main theorem, it suffices to show that $(ii) \Rightarrow (iv)$ and $(iii) \Rightarrow (iv)$. Indeed, $(i) \Rightarrow (ii)$, $(iv) \Rightarrow (i)$ and $(iv) \Rightarrow (iii)$ are clear. 

\section{Background material}
Let $f$ be a Dirichlet series of form $(1)$. We do not recall the definition of abscissa of simple (resp. absolute) convergence denoted by $\sigma_c$ (resp. $\sigma_a$), see \cite{queffelecBook} or \cite{tenenbaum1995introductiona} for more details. We shall need the two other following abscissas:
\[ \left. \begin{array}{ccl}
\sigma_u(f) & = &\inf \lbrace a \, | \hbox{ The series} \,(1) \hbox{ is uniformly convergent for } \Re(s)>a \rbrace   \\
            & = & \hbox{abscissa of uniform convergence of } f.   \\
            &   &                                                 \\   
\sigma_b(f) & = &\inf \lbrace a \, | \hbox{ the function } f \hbox{ has an analytic, bounded extension for } \Re(s)>a \rbrace   \\
            & = & \hbox{abscissa of boundedness of } f.  
\end{array}\right. \]

It is easy to see that $\sigma_c(f) \leq \sigma_u(f) \leq \sigma_a(f)$. An important result is that $\sigma_u(f)$ and $\sigma_b(f)$ coincide: this is the Bohr's theorem (see \cite{bohr1913uber}). This result is really useful for the study of $ \mathcal{H}^{ \infty}$, the algebra of bounded Dirichlet series on the right half-plane $\mathbb{C}_+$ (see \cite{maurizi2010some}). We shall denote by $ \Vert \,\cdot \,\Vert_{ \infty}$ the norm on this space:
\[ \Vert f \Vert_{ \infty}:= \sup_{ \Re(s)>0} \vert f(s) \vert. \]

We shall make a crucial use of the point evaluation in the proof of the Main Theorem: for every $p \geq 1$, the spaces $\mathcal{H}^p$ and $\mathcal{A}_{\mu}^p$ are spaces of holomorphic functions on $\mathbb{C}_{1/2}$ and more precisely if $\delta_s$ is the operator of point evaluation at $s \in \mathbb{C}_{1/2}$, then by \cite{bayart2002hardy},Th3:
\[ \Vert \delta_s \Vert_{(\mathcal{H}^p)^*} = \zeta(2 \Re(s)) \]
and by \cite{bailleullefevre2013},Th1 the point evaluation is also bounded on the spaces $\mathcal{A}_{\mu}^p$. Moreover $\sigma_b(f) \leq 1/2$ for any $f \in \mathcal{A}_{\mu}^p$. For example when $\mu= \mu_{\alpha}$, it is shown in \cite{bailleullefevre2013},Cor1 that there exists a positive constant $c_{\alpha,p}$ such that for every $s \in \mathbb{C}_{1/2}$,
\[ \Vert \delta_s \Vert_{ ( \mathcal{A}_{\alpha}^p )^*} \leq c_{\alpha,p} \bigg{ ( } \frac{\Re(s)}{2\Re(s)-1} \bigg{)}^{ \frac{2+\alpha}{p}} \cdot \]
When $p=2$, $\mathcal{A}_{\mu}^2$ is a Hilbert space and it is easy to see that
\[ \Vert f \Vert_{ \mathcal{A}_{\mu}^2} =  \bigg{(} \sum_{n=1}^{ + \infty} \vert a_n \vert^2 w_h(n) \bigg{)}^{1/2} \]
where for every $n \geq 1$,
\[ w_h(n) = \int_{0}^{ + \infty} n^{-2 \sigma} h( \sigma) d \sigma. \]
Thanks of the boundedness of the point evaluation at $s \in \mathbb{C}_{1/2}$, we consider the following reproducing kernels defined for every $w \in \mathbb{C}_{1/2}$ by
\[ K_{\mu}(s,w)= \sum_{n=1}^{+\infty} \frac{n^{- \overline{s}-w}}{w_h(n)} \cdot \]
For every $f \in \mathcal{A}_{\mu}^2$ and $s \in \mathbb{C}_{1/2}$, one has
\[ f(s) = < f, K_{\mu}(s, \, \cdot \, )>_{\mathcal{A}_{\mu}^2} \cdot \]

\begin{ex}
On the space $\mathcal{A}_{\alpha}^2$, we simply denote $(w_n^{\alpha})$ the corresponding weight and then for every $n \geq 1$,
\[ w_{n}^{\alpha} = \frac{1}{(\log(n)+1)^{\alpha+1}} \cdot \]
\end{ex}

Let $ \Phi: \mathbb{C}_{ \frac{1}{2}} \rightarrow \mathbb{C}_{ \frac{1}{2}}$ be an analytic function such that $\Phi(s)=c_0s+ \varphi(s)$ where $c_0$ is a nonnegative integer and $\varphi \in \mathcal{D}$. We shall say that $\Phi$ is a symbol if $C_{\Phi}$ is bounded on the spaces $\mathcal{A}_{\mu}^p$. \\

For $\sigma>0$, we denote $\Phi_{\sigma}$ the translate of $\Phi$ by $\sigma$: $\Phi_{\sigma}(s):= \Phi( \sigma+s)$. \\

When $c_0 \geq 1$, thanks to the Theorem $1$ we know that $\Phi$ is a symbol if and only $\varphi$ converges uniformly on $\mathbb{C}_{\varepsilon}$ for every $\varepsilon>0$ and $\varphi( \mathbb{C}_+) \subset \mathbb{C}_+$. In this case, it is easy to see that for every $\sigma>0$, $\Phi_{\sigma}- \sigma$ is also a symbol: indeed let $\sigma>0$ and $s \in \mathbb{C}_+$, then
\[ \begin{array}{ccl}
\Re( \Phi_{\sigma}(s)- \sigma)& =& \Re(c_0( \sigma+s)) + \Re( \varphi( \sigma+s)) - \sigma \\
& > & \sigma(c_0-1) + \Re(s)>0  \\
\end{array} \]
because $\varphi( \mathbb{C}_+) \subset \mathbb{C}_+$ and $s \in \mathbb{C}_+$. Point out that this result can be seen as the Schwarz's lemma in this framework. \\


\section{Proof of $(ii) \Rightarrow (iv)$}
With help of Proposition 4.2 from \cite{gordon1999composition}, F. Bayart proved the following useful lemma.

\begin{lemme}[\cite{bayart2002hardy},Lem11]
Let $\Phi$ be a symbol. If $\Phi$ is not a vertical translation then there exists $\varepsilon$ and $\eta>0$ such that
\[ \Phi( \mathbb{C}_{1/2- \varepsilon}) \subset \mathbb{C}_{1/2 + \eta}. \]
\end{lemme}

\begin{proof}[Proof of $(ii) \Rightarrow (iv)$]
We follow ideas from \cite{bayart2002hardy},Th14. Assume that $\Phi$ is not a vertical translation. By the previous lemma, there exists $\varepsilon$ and $\eta>0$ such that
\[ \Phi( \mathbb{C}_{1/2- \varepsilon}) \subset \mathbb{C}_{1/2 + \eta}.\]
We remark that each element of Im($C_{\Phi}$) is defined and bounded on $\mathbb{C}_{1/2- \varepsilon}$: indeed $\Phi (\mathbb{C}_{1/2- \varepsilon}) \subset \mathbb{C}_{1/2+ \eta}$ and if $f \in \mathcal{A}_{\mu}^p$, $f$ is bounded on $\mathbb{C}_{1/2 + \eta}$ (because $\sigma_b(f) \leq 1/2$).

Now by lemma \cite{bayart2002hardy},Lem9 we know that there exists $f \in \mathcal{H}^p$ such that the line $\Re(s)=1/2$ is both abscissa of convergence and natural boundary for $f$. Because of the inclusion $\mathcal{H}^p \subset \mathcal{A}_{\mu}^p$, $f$ belongs to $\mathcal{A}_{\mu}^p$. We consider the following infinite dimensional subspace of $\mathcal{A}_{\mu}^p$:
\[ F = \hbox{span} \lbrace n^{-s} f, \, n \geq 1 \rbrace = f \mathcal{P}. \]
We shall show that $F \cap Im( C_{\Phi}) = \lbrace 0 \rbrace$ and consequently Codim(Im($C_{\Phi}$))=$+ \infty$ which is a contradiction with $(ii)$.

Let $h \in F \cap Im( C_{\Phi})$, there exists $P \in \mathcal{P}$ such that $h=Pf$. If $h \neq 0$, there exists $s_0$ such that $\Re(s_0)=1/2$ and $P(s_0) \neq 0$. But in this case, $f$ extends beyond $\mathbb{C}_{1/2}$ and then we obtain a contradiction because the line $\Re(s)=1/2$ is a natural boundary for $f$. Finally $F \cap Im( C_{\Phi}) = \lbrace 0 \rbrace$.
\end{proof}

\section{Proof of $(iii) \Rightarrow (iv)$}
First we shall show that if $C_{\Phi}$ is an isometry then $c_0 \geq 1$. We need the following result.
\begin{lemme}

$  \dis\lim_{\Re(s) \rightarrow + \infty } \Vert \delta_{s} \Vert_{ (\mathcal{A}_{\mu}^1)^*} = 1.$

\end{lemme}

\begin{proof}
Let $s \in \mathbb{C}_1$. By the reproducing kernel property on $\mathcal{A}_{\mu}^2$ (or just by a simple computation), for any Dirichlet polynomial we have
\[ P(s) = \int_{0}^{+ \infty} \lim_{T \rightarrow + \infty} \frac{1}{2T} \int_{-T}^{T}  P(\sigma+it) \overline{K_{\mu}(s, \sigma+it)} \, dt d \mu( \sigma). \]
Now by definition of the norm of Dirichlet polynomials in $\mathcal{H}^1$ (see definition 1 from \cite{bayart2002hardy}), we have
\[ \Vert P \Vert_{ \mathcal{A}_{\mu}^1} = \bigg{(} \int_{0}^{+ \infty} \lim_{T \rightarrow + \infty} \frac{1}{2T} \int_{-T}^{T}  \vert P(\sigma+it)  \vert dt d \mu( \sigma) \bigg{)} \cdot \]
Consequently 
\[ \vert P(s) \vert \leq \Vert P \Vert_{ \mathcal{A}_{\mu}^1} \times \Vert K_{\mu}(s, \, \cdot \,) \Vert_{\infty}. \]
Now
\[ \Vert K_{\mu}(s, \cdot) \Vert_{\infty} = \sup_{w \in \mathbb{C}_+} \bigg{\vert} \sum_{n=1}^{+ \infty} \frac{n^{- \overline{s}+w}}{w_h(n)} \bigg{\vert} \leq \sum_{n=1}^{+ \infty} \frac{n^{- \Re(s)}}{w_h(n)}\]
and we point out that $w_h(1)=1$ so
\[ \limsup_{\Re(s) \rightarrow + \infty} \Vert \delta_{s} \Vert_{ (\mathcal{A}_{\mu}^1)^*} \leq \lim_{\Re(s) \rightarrow + \infty} \sum_{n=1}^{+ \infty} \frac{n^{- \Re(s)}}{w_h(n)} = 1. \]
On the other hand, it is clear that $\Vert \delta_s \Vert_{(\mathcal{A}_{\mu}^1)^{*}} \geq 1$ and then we obtain the result. 

\end{proof}

\begin{proposition}
Let $\Phi$ be a symbol. If $C_{\Phi}$ is a contraction then $c_0 \geq 1$.
\end{proposition}

\begin{proof}
Let $s \in \mathbb{C}_{1/2}$. For every $f \in \mathcal{A}_{\mu}^p$ we have
\[ \vert f \circ \Phi(s) \vert \leq \Vert \delta_s \Vert_{( \mathcal{A}_{\mu}^p)^*} \Vert f \circ \Phi \Vert \leq \Vert \delta_s \Vert_{( \mathcal{A}_{\mu}^p)^*} \Vert C_{\Phi} \Vert \Vert f \Vert \]
and then
\[ \frac{\Vert \delta_{ \Phi(s)} \Vert_{( \mathcal{A}_{\mu}^p)^*}}{\Vert \delta_s \Vert_{( \mathcal{A}_{\mu}^p)^*}} \leq \Vert C_{\Phi} \Vert. \]
By inclusion of the spaces $ \mathcal{A}_{\mu}^p$ and the fact that $\mathcal{H}^p \subset \mathcal{A}_{\mu}^p$ with $\dis\Vert \cdot \Vert_{ \mathcal{A}_{\mu}^p} \leq \Vert \cdot \Vert_{\mathcal{H}^p}$ we obtain:
\[ \frac{\Vert \delta_{ \Phi(s)} \Vert_{( \mathcal{H}^p)^*}}{\Vert \delta_s \Vert_{( \mathcal{A}_{\mu}^1)^*}} \leq \Vert C_{\Phi} \Vert. \]
By Theorem 3 from \cite{bayart2002hardy}, we obtain
\[ \zeta(2 \Re( \Phi(s)))^{1/p} \times \Vert \delta_s \Vert_{( \mathcal{A}_{\mu}^1)^*}^{-1}  \leq \Vert C_{ \Phi} \Vert. \]
Now assume $c_0=0$, then $\Phi(s)= \varphi(s) = \displaystyle{\sum_{n=1}^{+ \infty} c_n n^{-s}}$ and $\Re(c_1)>1/2$ (see proof of Lemma $3.3$ from \cite{gordon1999composition}). Finally thanks to the Lemma $2$, when $\Re(s)$ goes to infinity we get
\[ \Vert C_{\Phi} \Vert \geq \zeta(2 \Re(c_1))^{1/p} >1 \]
and consequently $C_{\Phi}$ is not a contraction.
\end{proof}

\begin{rem}
In the previous Lemma we actually used that for every $s \in \mathbb{C}_{1/2}$, $\delta_s \circ C_{\Phi} = \delta_{ \Phi(s)}$.
\end{rem}

\begin{proof}[Proof of $(iii) \Rightarrow (iv)$]
Assume that $C_{\Phi}$ is an isometry. By the last lemma, $c_0 \geq 1$ and then we know that $\Phi : \mathbb{C}_+ \rightarrow \mathbb{C}_+$ thanks to the Theorem $1$. One has
\[ \Vert 2^{-s} \Vert_{ \mathcal{A}_{\mu}^p} = \Vert 2^{- \Phi} \Vert_{\mathcal{A}_{\mu}^p}. \]
Now by \cite{bailleullefevre2013},Th6,
\[ \int_{0}^{ + \infty} \Vert 2^{- \sigma - \tiny\bullet} \Vert_{ \mathcal{H}^p}^p - \Vert 2^{- \Phi(\sigma + \tiny\bullet)} \Vert_{ \mathcal{H}^p}^p \, d \mu (\sigma) = 0. \]
But
\[ \Vert 2^{- \Phi(\sigma + \tiny\bullet)} \Vert_{ \mathcal{H}^p} = \Vert 2^{- \sigma - (\Phi(\sigma + \tiny\bullet)- \sigma)} \Vert_{ \mathcal{H}^p} = \Vert C_{ \Phi_{\sigma}- \sigma} (2^{ - \sigma - \bullet}) \Vert_{ \mathcal{H}^p}. \]
Thanks to the Schwarz's lemma in this framework (recall that $c_0 \geq 1$) we know that $\Phi_{\sigma}- \sigma : \mathbb{C}_+ \rightarrow \mathbb{C}_+$. So by the Theorem 1, $C_{\Phi_{\sigma}- \sigma}$ is a bounded composition operator on $\mathcal{H}^p$ and $\Vert C_{\Phi_{\sigma}- \sigma} \Vert \leq 1$. Then
\[  \Vert 2^{- \Phi(\sigma + \tiny\bullet)} \Vert_{ \mathcal{H}^p}  \leq \Vert 2^{ - \sigma - \bullet} \Vert_{ \mathcal{H}^p}. \]
Consequently $2^{- \sigma} = \Vert 2^{- \sigma - \tiny\bullet} \Vert_{ \mathcal{H}^p} = \Vert 2^{- \Phi(\sigma + \tiny\bullet)} \Vert_{ \mathcal{H}^p}$ for every $\sigma>0$ (recall that $h$ is a positive continuous function). Now by Lemma $1$, if $\Phi$ is not a vertical translation, there exists $\varepsilon$ and $\eta >0$ such that $\Phi( \mathbb{C}_{1/2- \varepsilon}) \subset \mathbb{C}_{1/2+ \eta}$ and then  for every $ \sigma > 1/2 -  \varepsilon$,
\[ 2^{-\sigma} = \Vert 2^{- \Phi(\sigma + \tiny\bullet)} \Vert_{ \mathcal{H}^p} \leq \Vert 2^{- \Phi(\sigma + \tiny\bullet)} \Vert_{ \mathcal{H}^{\infty}} \leq 2^{-1/2 - \eta} \]
and this is obviously false.
\end{proof}

\begin{rem}
Let $\mu$ be a probability measure on $(0,+ \infty)$ such that $0 \in Supp( \mu)$ and $d\mu= h d \sigma$ where $h$ is a nonnegative function. If there exists an open interval $I$ such that $h$ is positive on $I$ then the theorem still holds. It is a consequence of the following lemma and some easy adaptations of the previous proof.
\end{rem}

\begin{lemme}
Let $\Phi$ be a symbol with $c_0 \geq 1$. If $\Phi$ is not a vertical translation then for every $\varepsilon>0$, there exists $\eta=\eta_{\varepsilon}>0$ such that $\Phi( \mathbb{C}_{\varepsilon}) \subset \mathbb{C}_{\varepsilon+ \eta}$.
\end{lemme}

\begin{proof}
First we assume that $\varphi$ is non constant then $\varphi : \mathbb{C}_+ \rightarrow \mathbb{C}_+$ and by Proposition 4.2 from \cite{gordon1999composition}, there exists $\vartheta>0$ such that $\varphi(\mathbb{C}_{\varepsilon}) \subset \mathbb{C}_{\vartheta}$ and consequently $\Phi(\mathbb{C}_{\varepsilon}) \subset \mathbb{C}_{c_0 \varepsilon + \vartheta}$. In this case, it suffices to choose $\eta = (c_0-1) \varepsilon + \vartheta$ which is positive because $c_0 \geq 1$.

If $\varphi$ is constant equals to $i \tau$ ($\tau \in \mathbb{R})$ and $c_0>1$ then $\Phi( \mathbb{C}_{\varepsilon}) \subset \mathbb{C}_{c_0 \varepsilon}$ and it suffices to choose $\eta = (c_0-1) \varepsilon$.

If $\varphi$ is constant and equals to $c_1  \in \mathbb{C}_+$ and $c_0 \geq 1$, $\Phi( \mathbb{C}_{\varepsilon}) \subset \mathbb{C}_{ c_0 \varepsilon + \Re(c_1)}$ and it suffices to choose $\eta= (c_0-1) \varepsilon + \Re(c_1)$.
\end{proof}

\begin{footnotesize}
\nocite{*}
\bibliography{biblio2}
\bibliographystyle{plain}
\end{footnotesize}

{\small 
\noindent{\it 
Univ Lille-Nord-de-France UArtois, \\ 
Laboratoire de Math\'ematiques de Lens EA~2462, \\
F\'ed\'eration CNRS Nord-Pas-de-Calais FR~2956, \\
F-62\kern 1mm 300 LENS, FRANCE \\
maxime.bailleul@euler.univ-artois.fr
}}

\end{document}